\documentclass[12pt,a4paper,pdfps]{article}
\usepackage[cp1251]{inputenc}
\usepackage[english]{babel}
\usepackage{color}
\usepackage[colorlinks,linkcolor=blue,filecolor=blue,citecolor=blue]{hyperref}
\usepackage{amsmath,amssymb,amsthm}
\usepackage{graphicx}
\usepackage{comment}
\usepackage{enumitem}
\usepackage{subfigure,multirow}
\usepackage{geometry}
\theoremstyle{definition}
\newtheorem{definition}{Definition}
\newtheorem{remark}{Remark}
\newtheorem{example}{Example}

\theoremstyle{plain}
\newtheorem{corollary}{Corollary}
\newtheorem{lemma}{Lemma}
\newtheorem{theorem}{Theorem}
\newtheorem{proposition}{Proposition}

\newenvironment{enumerate*}%
  {\begin{enumerate}%
    \setlength{\itemsep}{1pt}%
    \setlength{\parskip}{1pt}}%
  {\end{enumerate}}

\newcommand{\Cone}{\mathcal{C}}
\newcommand{\R}{\mathbb{R}}
\DeclareMathOperator{\ri}{\mathrm{ri}}
\DeclareMathOperator{\rb}{\partial_{\mathrm{r}}}
\newcommand{\id}{\mathrm{id}}
\newcommand{\g}{\mathfrak{g}}
\DeclareMathOperator{\sspan}{\mathrm{span}}
\DeclareMathOperator{\const}{\mathrm{const}}
\DeclareMathOperator{\Ker}{\mathrm{Ker}}

\title{Sub-Lorentzian extremals defined by an antinorm
\footnote{The work is supported by the Russian Science Foundation under grant 22-21-00877 (https://rscf.ru/en/project/22-21-00877/) and performed in Ailamazyan Program Systems Institute of Russian Academy of Sciences.}}

\author{
A.\,V.~Podobryaev \\ A.\,K.~Ailamazyan Program Systems
Institute of RAS \\ \tt{alex@alex.botik.ru} \\
}

\date{}

\begin{document}

\maketitle

\begin{abstract}
We consider a left-invariant (sub-)Lorentzian structure on a Lie group. We assume that this structure is defined by a closed convex salient cone in the corresponding Lie algebra and a continuous antinorm associated with this cone. We derive the Hamiltonian system for (sub-)Lorentzian extremals and give conditions under that normal extremal trajectories keep their causal type. Tangent vectors of abnormal extremal trajectories are either light-like or tangent vectors of sub-Riemannian extremal trajectories for the sub-Riemannian distribution spanned by the cone.


\end{abstract}

\section{\label{sec-introduction}Introduction}

Recently, interest in left-invariant Lorentzian and sub-Lorentzian problems has grown from the point of view of geometric control theory.
Let us first mention the pioneering works in this direction by M.~Grochowski~\cite{1, 2} devoted to the sub-Lorentzian geometry of the Heisenberg group. This research was continued by Yu.\,L.~Sachkov and E.\,F.~Sachkova~\cite{3, 4}. See also the work~\cite{5} of E.~Grong and A.~Vasil'ev on left-invariant sub-Lorentzian geometry on the anti-de\,Sitter space and the work of Yu.\,L.~Sachkov~\cite{6, 7} on left-invariant Lorentzian geometry on Lobachevsky plane.

We use the Hamiltonian formalism of Pontryagin's maximum principle to describe (sub-)Lorentzian extremal trajectories. The difficulty is that the (sub-)Lorentzian length functional contains a square root, and the optimization problem is to maximize this functional, in contrast to the (sub-)Riemannian problem, which aims to minimize the corresponding length functional.
This is why the standard trick for sub-Riemannian geometry~\cite[Ch.~3]{8} of replacing the length functional with an energy functional (see, for example, \cite[\S~3.3.1]{9}) does not work in the (sub-)Lorentzian case.
Recall that this trick uses the Cauchy-Bunyakowski-Schwartz inequality
$$
l(g)^2 = \Biggl( \int\limits_0^{t_1}{\sqrt{q(\dot{g}(t))} \, dt} \Biggr)^2 \leqslant t_1 \cdot \int\limits_0^{t_1}{q(\dot{g}(t)) \, dt} = 2 t_1 J(g),
$$
where $g : [0, t_1] \rightarrow G$ is an admissible curve,
the length of its tangent vectors is defined by a quadratic form $q$,
and $l(g)$, $J(g)$ denote the length and the energy of the curve $g$, respectively.
It is important that the equality above is achieved only on constant speed curves.

Nevertheless, it turns out that the "energy" can be used in the (sub-)Lorentzian case to describe normal extremals. In addition, we consider a more general problem statement where a left-invariant sub-Lorentzian structure is defined by an arbitrary closed convex salient cone in the Lie algebra and a continuous antinorm associated with this cone.
Using Pontryagin's maximum principle, the corresponding Hamiltonian system is derived.
It turns out that under some additional conditions the normal extremal trajectory retains its causal type
(i.e. its tangent vector always remains either time-like or light-like).
More precisely, the dual function of the antinorm should be an antinorm as well (associated with the dual cone).
Moreover, unlike Riemannian geometry, abnormal extremal trajectories always arise in Lorentzian geometry. These abnormal trajectories coincide with light-like extremal trajectories, and thus are not strictly abnormal. In sub-Lorentzian geometry, generally speaking, abnormal extremal trajectories can have both light-like tangent vectors and tangent vectors coinciding with the tangent vectors of some of sub-Riemannian abnormal trajectories determined by the distribution of subspaces, which is linearly generated by the cone. If this distribution is contact, then abnormal extremal trajectories are light-like and not strictly abnormal.

It is appropriate to mention here the work of L.\,V.~Lokutsievskiy~\cite{10} where a remarkable technique of convex trigonometry was developed, which makes it possible to parameterize solutions of Hamiltonian systems for extremals in the sub-Finsler case with a two-dimensional distribution~\cite{11, 12}, as well as for some distributions of higher dimensions~\cite{13}. It would be interesting to develop concave hyperbolic trigonometry for sub-Lorentzian problems defined by an arbitrary antinorm.

This work has the following structure. In Section~\ref{sec-problem-statement}, after recalling the concept of an antinorm, the optimal control problem statement is given. Then, in Section~\ref{sec-extremal-trajectories} we introduce some necessary definitions from convex analysis and we derive a Hamiltonian system for extremal trajectories using Pontryagin's maximum principle. Section~\ref{sec-crl-ex} provides some corollaries and examples.

The author is grateful to the anonymous reviewer for the useful comments that helped to improve the article significantly.

\section{\label{sec-problem-statement}Problem statement}

Let $\Cone$ be a closed convex cone in a finite-dimensional real vector space $V$.

\begin{definition}
\label{def-1}
{\it The relative interior} of the cone $\Cone$ is the set
$$
\ri{\Cone} = \{u \in \Cone \, | \, \forall\,v\in\Cone, \, v \neq u \ \exists\,w\in\Cone, \lambda \in (0,1) : u = \lambda v + (1-\lambda)w\}.
$$
{\it The relative boundary} of the cone $\Cone$ is the set $\rb{\Cone} = \Cone \setminus \ri{\Cone}$.
\end{definition}

\begin{definition}
\label{def-2}
{\it An antinorm}~\cite{14} associated with the cone $\Cone$ is a function $\alpha : V \rightarrow \R \cup \{-\infty\}$ such that
\begin{enumerate}
\item[(i)] $\alpha|_{\ri{\Cone}} > 0$, $\alpha|_{\rb{\Cone}} = 0$, $\alpha|_{V \setminus \Cone} = -\infty$;
\item[(ii)] for any $v \in V$ and $\lambda > 0$ the equality $\alpha(\lambda v) = \lambda \alpha(v)$ is satisfied;
\item[(iii)] for any $v, w \in V$ the inequality $\alpha(v+w) \geqslant \alpha(v) + \alpha(w)$ is satisfied, i.e., the function $\alpha$ is concave.
\end{enumerate}
\noindent An antinorm $\alpha$ is called {\it continuous}, if the function $\alpha|_{\Cone}$ is continuous.
\end{definition}

Consider the following left-invariant optimal control problem on a real finite-dimensional Lie group $G$.
Let $\Cone \subset \g$ be a closed convex cone in the corresponding Lie algebra, and $\alpha$ be a continuous antinorm associated with the cone $\Cone$.
It is required to find a Lipschitz curve $g : [0, t_1] \rightarrow G$ connecting the unit element of the group $G$ with an arbitrary fixed element $g_1 \in G$,
and a measurable control $u \in L^{\infty}([0, t_1], \, \Cone \setminus 0)$ with values in the set $\Cone \setminus 0$ such that
\begin{equation}
\label{podobryaev-eq-optimal-control-problem}
g(0) = \id, \qquad g(t_1) = g_1, \qquad \dot{g}(t) = L_{g(t) *} u(t), \qquad \int\limits_0^{t_1}{\alpha(u(t)) \, dt} \rightarrow \max,
\end{equation}
where the terminal time $t_1$ is free, and $L_g$ denotes left-shift by an element $g \in G$.

\begin{remark}
\label{rem-name}
It is natural to call problem~\eqref{podobryaev-eq-optimal-control-problem} {\it the (sub-)Lorentzian-Finsler problem}.
Indeed, if the antinorm $\alpha$ is defined by a nondegenerate quadratic form of the signature $(1, n)$, for example,
$$
\alpha|_{\Cone}(u) = \sqrt{u_0^2 - u_1^2 - \dots - u_n^2}, \qquad
\Cone = \{u = (u_0,u_1,\dots,u_n) \in \R^{n+1} \, | \, u_0^2 \geqslant u_1^2 + \dots + u_n^2, \, u_0 \geqslant 0 \},
$$
then we get the maximisation problem for Lorentzian length in Minkowski space $\R^{1, n}$.

If $\dim{\Cone} < \dim{\g}$, then velocities of admissible curves are located in some distribution of subspaces in the tangent bundle of the Lie group $G$.
By analogy with the sub-Riemannian case we will call the corresponding problem sub-Lorentzian.

In addition, in the problem statement~\eqref{podobryaev-eq-optimal-control-problem} an arbitrary continuous antinorm is considered. In the case of an arbitrary norm the corresponding minimization problem is called the Finsler problem. So, in our case it is natural to call it the Lorentzian-Finsler problem.
\end{remark}

\begin{remark}
\label{rem-difficult}
What is the attainable set for problem~\eqref{podobryaev-eq-optimal-control-problem}?
Is there an optimal solution of this problem for the given boundary conditions?
Generally speaking, these questions are not trivial.
For example, in the Lorentz problem on anti-de Sitter space~\cite{5} the necessary optimality condition (Pontryagin's maximum principle) identifies a set where globally optimal solutions can exist, contained in a nontrivial attainable set. Moreover, globally optimal solutions exist but have bounded length.
In some sense the opposite example is provided by the sub-Lorentzian problem on the Heisenberg group~\cite{3}, where globally optimal solutions exist on the entire attainable set.
In this work only extremal trajectories are studied, i.e., trajectories satisfying the necessary optimality condition --- Pontryagin's maximum principle.
We refer to paper~\cite{lokutsievskiy-podobryaev} for an existence theorem of optimal solution for the problem statement considered here.
\end{remark}

\section{\label{sec-extremal-trajectories}Extremal trajectories}

Let us recall some necessary definitions from convex analysis. Below $V^*$ denotes the dual space of the vector space $V$ and $\langle \, \cdot \, , \, \cdot \, \rangle$ denotes the canonical pairing of covectors and vectors.

\begin{definition}
\label{def-3}
The cone $\Cone^{\vee} = \{p \in V^* \, | \, p|_{\Cone} \leqslant 0 \} \subset V^*$ is called {\it the negative dual cone} for the cone $\Cone$.
{\it An antisphere} of radius $r$ for an antinorm $\alpha$ is the set $S_r = \{v \in V \, | \, \alpha(v) = r\}$.
{\it The dual function} for the antinorm $\alpha$ is the function $\alpha^{\vee} : V^* \rightarrow \R \cup \{-\infty\}$ such that
$$
\alpha^{\vee}(p) = -\sup\limits_{v \in S_1}{\langle p, v \rangle}, \qquad p \in V^*.
$$
\end{definition}

\begin{definition}
\label{def-4}
A cone is called {\it salient}, if it does not contain any nonzero subspaces.
\end{definition}

\begin{lemma}
\label{lemma-1}
Let $\Cone$ be a salient cone. Then if $p \in \ri{(\Cone^{\vee})}$ we have $p|_{\Cone \setminus 0} < 0$.
\end{lemma}

\begin{proof}
Assume by contradiction that there exists nonzero $u \in \Cone$ such that $\langle p, u\rangle = 0$.
Since $p \in \ri{(\Cone^{\vee})}$, by Definition~\ref{def-1} for any $q \in \Cone^{\vee}$ there exist $r \in \Cone^{\vee}$ and $\lambda \in (0,1)$ such that
$p = \lambda q + (1-\lambda)r$. In particular, $\lambda \langle q, u \rangle + (1-\lambda)\langle r, u \rangle = \langle p, u \rangle = 0$.
Hence, $\langle q, u \rangle = 0$.
So, we get $\sspan{\{u\}} \subset \Cone^{\vee\vee} = \Cone$ in contradiction with the condition that the cone $\Cone$ is salient.
\end{proof}

\begin{lemma}
\label{lemma-2}
Assume that $\Cone$ is a salient cone and an antinorm $\alpha$ is continuous.
Then the dual function $\alpha^{\vee}$ is an antinorm associated with the dual cone $\Cone^{\vee}$ if and only if $\alpha^{\vee}|_{\rb{(\Cone^{\vee})}} = 0$.
\end{lemma}

\begin{proof}
It is clear that the function $\alpha^{\vee}$ is homogeneous and concave.
Moreover, $\alpha^{\vee}|_{\Cone^{\vee}} \geqslant 0$ and $\alpha^{\vee}|_{V^* \setminus \Cone^{\vee}} = -\infty$.
Let us prove that $\alpha^{\vee}|_{\ri{(\Cone^{\vee})}} > 0$.
Indeed, if $p \in \ri{(\Cone^{\vee})}$, then by Lemma~\ref{lemma-1} $p|_{\Cone \setminus 0} < 0$.
Antisphere $S_1$ is separated of the hyperplane $\{u \in V \, | \, \langle p, u \rangle = 0\}$ by the closed boundary of the cone $\Cone$.
It follows that $\alpha^{\vee}(p) > 0$.
Really, if $\alpha^{\vee}(p) = 0$, then for a sequence $\varepsilon_n > 0$, $\varepsilon_n \rightarrow 0$ there is a sequence $u_n \in S_1$ such that
$-\varepsilon_n < \langle p, u_n \rangle \leqslant 0$.
The sequence $u_n$ is separated from zero, since the antinorm $\alpha$ is continuous.
Then the sequence $\frac{1}{|u_n|}$ is bounded, where $|\cdot|$ is the Euclidian norm in the space $V$.
Whence, $\langle p, \frac{u_n}{|u_n|} \rangle \rightarrow 0$.
Since the cone $\Cone$ is closed, then passing to a subsequence we get $\frac{u_n}{|u_n|} \rightarrow u \in \Cone \setminus 0$ and $\langle p, u \rangle = 0$.
So, we have a contradiction.
To satisfy all the requirements of Definition~\ref{def-2}, the condition $\alpha^{\vee}|_{\rb{(\Cone^{\vee})}} = 0$ must be satisfied.
\end{proof}

Let us define an extremal trajectory of the problem~\eqref{podobryaev-eq-optimal-control-problem}.

\begin{definition}
\label{def-5}
Consider the family of functions $H^{\nu}_u : T^*G \rightarrow \R$ on the cotangent bundle of the Lie group $G$, that depends on parameters $u \in \Cone \setminus 0$ and $\nu \in \{0,1\}$,
defined as follows:
$$
H^{\nu}_{u}(\lambda) = \langle L_{\pi(\lambda)}^* \lambda, u \rangle + \nu \alpha(u), \qquad \lambda \in T^*G.
$$
A Lipschitz curve $\lambda : [0, t_1] \rightarrow T^*G$ is called {\it an extremal},
if $t_1 > 0$ and there exist an admissible control $\hat{u} \in L^{\infty}([0, t_1], \, \Cone \setminus 0)$ and a number $\nu \in \{0,1\}$ such that
$(\lambda, \nu) \neq 0$ and for almost all $t \in [0, t_1]$ the following conditions are satisfied:
\begin{equation}
\label{podobryaev-eq-pmp}
\dot{\lambda}(t) = \vec{H}^{\nu}_{\hat{u}(t)}(\lambda(t)), \qquad
H^{\nu}_{\hat{u}(t)}(\lambda(t)) = \max\limits_{u \in \Cone \setminus 0}{H^{\nu}_{u}(\lambda(t))}, \qquad
H^{\nu}_{\hat{u}(t)}(\lambda(t)) = 0,
\end{equation}
where $\vec{H}^{\nu}_{\hat{u}(t)}$ is the Hamiltonian vector field corresponding to the Hamiltonian $H^{\nu}_{\hat{u}(t)}$ with respect to the canonical symplectic structure of the cotangent bundle $T^*G$.

If $\nu = 1$, then the curve $\lambda$ is called {\it a normal extremal}. If $\nu = 0$, then this curve is called {\it an abnormal extremal}.
Let $\pi : T^*G \rightarrow G$ be the natural projection.
The curve $\pi \circ \lambda : [0, t_1] \rightarrow G$ is called {\it a normal/abnormal extremal trajectory}.
An abnormal extremal trajectory is {\it strictly abnormal} if it is not a projection of any normal extremal.
\end{definition}

\begin{remark}
\label{rem-extr}
If $(\hat{g}, \hat{u})$ is an optimal process for problem~\eqref{podobryaev-eq-optimal-control-problem},
then in accordance to the Pontryagin maximum principle (see classical book~\cite[\S~3]{15} or a more modern presentation~\cite[Ch.~12]{16})
the curve $\hat{g}$ is the extremal trajectory corresponding to the control $\hat{u}$.
The conditions of the Pontryagin maximum principle (see, for example, \cite[\S~10.1]{16}) are satisfied automatically in our situation.
Namely, the left-invariance of problem~\eqref{podobryaev-eq-optimal-control-problem} on the Lie group implies smoothness of the admissible vector fields,
and continuity of the antinorm $\alpha$ implies continuity in the control of the quality functional.
\end{remark}

Note that in the case of an arbitrary continuous antinorm $\alpha$ the maximized Hamiltonian is, generally speaking, nonsmooth.
Therefore, the equations $\dot{\lambda}(t) = \vec{H}^{\nu}_{\hat{u}(t)}(\lambda(t))$ should be understood as a family of Hamiltonian vector fields.
In other words, the Hamiltonian system is non-autonomous and depends on control.

\begin{remark}
\label{rem-vert}
Since the Hamiltonians $H^{\nu}_{\hat{u}(t)}$ are left invariant, the Hamiltonian vector fields $\vec{H}^{\nu}_{\hat{u}(t)} $ are determined by their vertical components.
More precisely, we can assume~\cite[\S~18.3]{16} that the functions $H^{\nu}_{\hat{u}(t)}$ are defined on the dual space of the Lie algebra $\g^* = T^ *_{\id}G$
with coordinates $h_1 = \langle\,\cdot\, , e_1\rangle, \dots, h_n = \langle\,\cdot\, , e_n\rangle$, where $e_1,\dots,e_n$ is some basis of the space $\g$.
Then the extremal $\lambda(t)$ is determined by the conjugate subsystem (of the Hamiltonian system)
$\dot{h}_i(t) = \{H^{\nu}_{\hat{u}(t)}, h_i(t)\}$ on the space $\g^*$, where $h (t) = (h_1(t),\dots,h_n(t)) = L^*_{\pi(\lambda(t))} \lambda(t)$, and $\{\,\cdot\ , , \,\cdot\,\}$ is the standard Poisson structure on the space $\g^*$.
\end{remark}

Let $p \in \g^*$ be a covector. Introduce the notation $u_p = \arg\max\limits_{u \in \Cone \setminus 0}{H^{\nu}_u(p)}$.
Generally speaking, $u_p$ is not uniquely defined. The covector $h(0) \in \g^*$ and the choice of control values $u(t) = u_{h(t)}$ uniquely determines the extremal $\lambda(t)$ with the initial condition $\lambda(0) = h(0)$.

\begin{definition}
\label{def-6}
We will call a tangent vector at a point $g \in G$ {\it time-like} (respectively, {\it light-like}) if it is located in $L_{g *}\ri{\Cone}$
(respectively, in $L_{g *}\rb{\Cone}$). A trajectory is called {\it timelike}/{\it lightlike} if each of its tangent vectors is timelike/lightlike.
These terms come from Minkowski geometry and special relativity.
\end{definition}

\begin{theorem}
\label{th-1}
Consider optimal control problem~\eqref{podobryaev-eq-optimal-control-problem} defined by a closed convex salient cone $\Cone$ and an associated continuous antinorm $\alpha$.
Assume that the dual function $\alpha^{\vee} $ is an antinorm associated with the dual cone $\Cone^{\vee}$.
Then every extremal trajectory $g(\cdot)$ is a solution of the equation $\dot{g}(t) = L_{g(t) *} u_{h(t)}$, where
$\dot{h}_i(t) = \{H_{u_{h(t)}}, h_i(t)\}$ and $H_{u_{h(t)}} = \langle h(t) , u_{h(t)} \rangle$.
\medskip

\emph{(1)} For a normal extremal trajectory one of the two conditions is satisfied\emph{:}
\begin{enumerate}
\item[\emph{(a)}] $h(t) \in S^{\vee}_1 = \{ p \in \g^* \, | \, \alpha^{\vee}(p) = 1 \}$ for any $t$ and the trajectory is time-like;
\item[\emph{(b)}] $h(t) \in S^{\vee}_0 = \{ p \in \g^* \, | \, \alpha^{\vee}(p) = 0 \}$ for any $t$ and the trajectory is light-like.
\end{enumerate}

\emph{(2)} Tangent vectors of abnormal extremal trajectories are either light-like vectors or tangent vectors
of sub-Riemannian abnormal trajectories determined by the distribution of subspaces $L_{g *} \sspan{\Cone} \subset T_gG$.
In particular, light-like arcs are not strictly abnormal.
\end{theorem}

The proof of Theorem~\ref{th-1} requires several additional lemmas.

For a covector $p \in V^*$ define the following sets:
\begin{equation}
\label{podobryaev-eq-def-dual}
p^{\vee}_r = \{v = \arg\,\max\limits_{u \in S_r}{\langle p, u \rangle}\}, \qquad p^{\vee} = \bigcup\limits_{r \geqslant 0}{p^{\vee}_r}.
\end{equation}

\begin{lemma}
\label{lemma-3}
Let $\Cone$ be a salient cone. Then if $p \in \ri{(\Cone^{\vee})}$ we have $p^{\vee}_0 = 0$.
\end{lemma}

\begin{proof}
If $p^{\vee}_0 \neq 0$, then there exists $u \in \rb{\Cone} \setminus 0$ such that $\langle p, u \rangle = 0$, in contradiction with Lemma~\ref{lemma-1}.
\end{proof}

\begin{lemma}
\label{lemma-4}
If an antinorm $\alpha$ is continuous, then for any $p \in \Cone^{\vee}$ the set $p^{\vee}$ is a closed convex cone.
\end{lemma}

\begin{proof}
It is clear that since the antinorm $\alpha$ is homogeneous, then the set $p^{\vee}$ is a cone. Moreover, $p^{\vee}_r = r p^{\vee}_1$ for $r > 0$.

Let $M = \sup\limits_{u \in S_1}{\langle p, u \rangle}$. It follows that $M \leqslant 0$ and the function $F(u) = \langle p, u \rangle - M \alpha(u)$ is a continuous homogeneous concave function on the cone $\Cone$. Obviously, $F|_{S_1} \leqslant 0$. So, by homogeneity we get $F|_{\ri{\Cone}} \leqslant 0$.
Since the function $F$ is continuous, we have $F|_{\Cone} \leqslant 0$.
The function $F$ is concave, this implies that the set $D = \{ u \in V \, | \, F(u) \geqslant 0\}$ is convex and closed.
Moreover, the function $F|_{\Cone}$ vanishes exactly on the set $p^{\vee}$.
So, the set $p^{\vee} = \Cone \cap D$ is convex and closed as an intersection of convex and closed sets.
\end{proof}

\begin{lemma}
\label{lemma-5}
If $p \in \Cone^{\vee}$ and $p|_{\Cone} \neq 0$, then $p|_{\ri{\Cone}} < 0$.
\end{lemma}

\begin{proof}
Assume by contradiction that there exists $v \in \ri{\Cone}$ such that $\langle p, v \rangle = 0$.
Then by Definition~\ref{def-1} for $w_1 \in \Cone$, $\langle p, w_1 \rangle \neq 0$ there exists $w_2 \in \Cone$ such that
$v$ is in the interior of the segment connecting the points $w_1$ and $w_2$.
It follows that the linear function $p$ takes values of opposite signs at the points $w_1, w_2 \in \Cone$, in contradiction with $p \in \Cone^{\vee}$.
\end{proof}

\begin{lemma}
\label{lemma-6}
If an antinorm $\alpha$ is continuous and $p \in \Cone^{\vee}$, $p|_{\Cone} \neq 0$, then this antinorm $\alpha$ is linear on the cone $p^{\vee}$.
\end{lemma}

\begin{proof}
It is sufficient to prove that $\alpha(u+v) = \alpha(u) + \alpha(v)$ for any $u, v \in p^{\vee}$, since the antinorm $\alpha$ is homogeneous.

If $u, v \in p^{\vee}_0$, this is obvious.
If $u, v \notin p^{\vee}_0$, then
$$
\alpha(u + v) = \alpha(\alpha(u) \bar{u} + \alpha(v) \bar{v}) = (\alpha(u)  + \alpha(v))\alpha(\lambda\bar{u} + \mu \bar{v}),
$$
$$
\bar{u} = \frac{u}{\alpha(u)}, \qquad \bar{v} = \frac{v}{\alpha(v)}, \qquad
\lambda = \frac{\alpha(u)}{\alpha(u)  + \alpha(v)}, \qquad
\mu = \frac{\alpha(u)}{\alpha(u)  + \alpha(v)}.
$$
Lemma~\ref{lemma-4} implies $\lambda\bar{u} + \mu\bar{v} \in p^{\vee}_1$ and we get the required equality.

It remains to consider the case $\alpha(u) \neq 0$, $\alpha(v) = 0$. By Lemma~\ref{lemma-4} we may assume that $\alpha(u) = 1$.
Moreover, by Lemma~\ref{lemma-5} we have $\langle p, u \rangle < 0$.

Note that $\langle p, v \rangle = 0$. Indeed, in the opposite case for $\lambda \in (0, 1)$ we have
$\langle p, \lambda v \rangle > \langle p, v \rangle$ and $\alpha(\lambda v) = 0$, it follows that $v \notin p_0^{\vee}$.

Next, since the antinorm is concave we get $\alpha(u+v) \geqslant \alpha(u) + \alpha(v) = 1$.
Assume that $\alpha(u+v) > 1$, then there exists $\lambda \in (0,1)$ such that $\alpha(\lambda(u+v)) = 1$.
Moreover, $\langle p, \lambda(u+v) \rangle = \lambda \langle p, u \rangle > \langle p, u \rangle$,
in contradiction with $u \in p_1^{\vee}$. So, we have $\alpha(u+v) = 1$.
\end{proof}

\begin{lemma}
\label{lemma-7}
If $p \in \rb{(\Cone^{\vee})}$ and $p|_{\Cone} \neq 0$,
then $p|_{\ri{\Cone}} < 0$ and there exists nonzero $u \in \rb{\Cone}$ such that $\langle p, u \rangle = 0$.
\end{lemma}

\begin{proof}
Lemma~\ref{lemma-5} implies $p|_{\ri{\Cone}} < 0$.
Assume by contradiction that $p|_{\rb{\Cone} \setminus 0} < 0$, then $p|_{\Cone \setminus 0} < 0$.
Take arbitrary $q \in \Cone^{\vee}$.
Assume that for any $\varepsilon > 0$ there exists $v \in \Cone$ such that
$\langle p+\varepsilon(p-q), v \rangle > 0$.
This implies that for a sequence $\varepsilon_n > 0$, $\varepsilon_n \rightarrow 0$ there exists a sequence $v_n \in \Cone \setminus 0$ such that
\begin{equation}
\label{podobryaev-eq-inequality}
\langle p, v_n \rangle > \frac{\varepsilon_n}{1 + \varepsilon_n}\langle q, v_n \rangle.
\end{equation}
Since the functions $p$ and $q$ are linear, we may assume that the sequence $v_n$ is bounded and separated from zero.
The cone $\Cone$ is closed, so, there is a subsequence that converges to an element $v \in \Cone \setminus 0$.
Inequality~\eqref{podobryaev-eq-inequality} implies $\langle p, v \rangle \geqslant 0$.
But $p|_{\Cone \setminus 0} < 0$. Hence, for $q \in \Cone^{\vee}$ there exists $\varepsilon > 0$ such that $(p+\varepsilon(p-q))|_{\Cone} \leqslant 0$,
i.e., $p+\varepsilon(p-q) \in \Cone^{\vee}$.
By Definition~\ref{def-1} we have $p \in \ri{\Cone}$. We get a contradiction.
\end{proof}

\begin{proof}[Proof of Theorem~\emph{\ref{th-1}}]

(1) Assume that $\nu = 1$.
Let us prove that a covector $p \in \g^*$ defines a corresponding control $u_p$ if and only if
$p \in S_1^{\vee}$ or $p \in S_0^{\vee}$, $p|_{\Cone} \neq 0$ and in this case the control $u_p$ is time-like or light-like, respectively.

Consider the case $p \in \ri{(\Cone^{\vee})}$, then by Lemma~\ref{lemma-3} we have $p^{\vee}_0 = 0$. Equalities~\eqref{podobryaev-eq-def-dual} imply
$\arg\max\limits_{u \in S_r}{H^{1}_u} = \arg\max\limits_{u \in S_r}{\langle p, u \rangle} = p^{\vee}_r$.
Hence, by Lemmas~\ref{lemma-3}--\ref{lemma-4} the corresponding control
\begin{equation}
\label{podobryaev-eq-mu}
u_p = \arg\max\limits_{u \in p^{\vee} \setminus 0}{(\langle p, u \rangle + \alpha(u))} = \bar{u} \cdot \arg\max\limits_{\mu > 0}{\mu(\langle p, \bar{u} \rangle + 1)},
\end{equation}
where $\bar{u} \in p^{\vee}_1$.

If $\langle p, \bar{u} \rangle + 1 \neq 0$, then there is no maximum with respect to the variable $\mu > 0$.
If $\langle p, \bar{u} \rangle + 1 = 0$, then from Definition~\ref{def-3} it follows that $p \in S^{\vee}_1$.
Then $\mu$ can be any positive number and $u_p \in p^{\vee} \setminus 0 \subset \ri{\Cone}$.
So, $p \in S^{\vee}_1$ and the corresponding control $u_p$ is time-like.

Consider now the case $p \in \rb{(\Cone^{\vee})}$ and $p|_{\Cone} \neq 0$. From Lemma~\ref{lemma-7} it follows that $p^{\vee}_0 \neq 0$.
Let us prove that $p^{\vee} = p^{\vee}_0$.
If $p^{\vee} \neq p^{\vee}_0$, then Lemma~\ref{lemma-4} implies $p^{\vee}_1 \neq \varnothing$.
Whence, $\alpha^{\vee}(p) \neq 0$. But due to the condition of the theorem the function $\alpha^{\vee}$ is an antinorm, then by Lemma~\ref{lemma-2} we have $\alpha^{\vee}(p) = 0$. We get a contradiction.
So, $p^{\vee} = p_0^{\vee}$ and the control $u_p$ is light-like.

It remains to note that in the case $p \notin \Cone^{\vee}$ or $p|_{\Cone} = 0$ there is no maximum of the expression $H^1_u$ with respect to the variable $u \in \Cone \setminus 0$.

So, the trajectory $h(\cdot)$ of the conjugate subsystem must lie in the set $S_1^{\vee} \cup S_0^{\vee}$.
If $h(t) \in S_1^{\vee}$, then the corresponding tangent vector of the extremal trajectory is time-like.
If $h(t) \in S_0^{\vee}$, then this tangent vector is light-like.

Notice that the antinorm $\alpha^{\vee}$ is upper semi-continuous.
In particular, the set $S^{\vee}_{\geqslant 1} = \{p \in \g^* \, | \, \alpha^{\vee}(p) \geqslant 1\}$ is closed.
The continuous curve $h(t)$ lies either in the closed set $S^{\vee}_{\geqslant 1}$ or in the closed set $S^{\vee}_0$.
It follows that it is located either in the set $S_1^{\vee}$ or in the set $S_0^{\vee}$.
Since $\alpha^{\vee}(h(t)) = \const$, this implies the alternative (a) or (b) from the theorem statement.

It remains to note that we can replace the Hamiltonian $\langle h(t), u_{h(t)} \rangle + \alpha(u_{h(t)})$ by the Hamiltonian $H_{u_{h(t)}} = \langle h(t), u_{h(t)} \rangle$.
\medskip

(2) Assume now that $\nu = 0$. The Hamiltonian equals $H^0_u(p) = \langle p, u \rangle$, it has a maximum with respect to the variable $u \in \Cone \setminus 0$ if and only if $p \in \rb{(\Cone^{\vee})}$.

If $p \in \rb{(\Cone^{\vee})}$ and $p|_{\Cone} \neq 0$, then this maximum equals zero and is achieved only on the set $p^{\vee}_0 \setminus 0$ by Lemma~\ref{lemma-7}.

If $p|_{\Cone} = 0$, then this maximum is achieved at any point of the set $\Cone \setminus 0$.

In the first case the tangent vector of the extremal trajectory is light-like.
In the second case the tangent vector coincides with a tangent vector of a sub-Riemannian abnormal trajectory.
In addition $\alpha^{\vee}(h(t)) \equiv 0$. This implies that each tangent vector of an abnormal trajectory is either light-like
or coincides with a tangent vector of a sub-Riemannian abnormal trajectory.
\end{proof}

\section{\label{sec-crl-ex}Corollaries and examples}

Let us give some examples showing the importance of the conditions of Theorem~\ref{th-1}.

\begin{remark}
\label{rem-zero-control}
If in problem~\eqref{podobryaev-eq-optimal-control-problem} we assume control equal to zero, then
under the conditions of Theorem~\ref{th-1} a normal extremal trajectory will be time-like up to parametrisation if and only if
for the trajectory $h(\cdot)$ of the conjugate subsystem the inequality $\alpha^{\vee}(h(t)) \geqslant 1$ holds for all $t$.
This follows easily from the closedness of the set $S_{\geqslant 1}^{\vee}$
and from the case $p \in \ri{(\Cone^{\vee})}$ of part~(1) of the proof of Theorem~\ref{th-1}.
Namely, for $\langle p, \bar{u} \rangle + 1 < 0$ the maximum of expression~\eqref{podobryaev-eq-mu} on the set $\mu \geqslant 0$ is achieved at the point $\mu = 0$.
\end{remark}

\begin{remark}
\label{rem-salient}
We assume that the cone $\Cone$ is salient. This condition is essential for Theorem~\ref{th-1}.
Indeed, if the cone $\Cone$ contains some nonzero subspace $W$, then for any $p \in \Cone^{\vee}$ we have $p|_W = 0$.
Otherwise there exists $w \in W$ such that $\langle p, w \rangle < 0$, then $\langle p, -w \rangle > 0$, but $-w \in \Cone$, we get a contradiction.
This means that $p^{\vee}_0 \supset W$. Then the extremal trajectory with the initial covector $p$ can have both time-like tangent vectors and light-like ones.
\end{remark}

\begin{definition}
\label{def-7}
We call a continuous antinorm $\alpha$ associated with a closed convex cone $\Cone$ {\it linear on a subcone adjacent to the boundary}
if there exists $p \in \rb{(\Cone^{\vee})}$ such that $p|_{\Cone} \neq 0$ and $\ri{(p^{\vee})} \subset \ri{\Cone}$
(by Lemma~\ref{lemma-6} the function $\alpha|_{p^{\vee}}$ is linear).
\end{definition}

\begin{example}
\label{ex-1}
Consider in the space $\R^2$ with coordinates $(x, y)$ the cone $\Cone = \{(x, y) \, | \, y \geqslant |x|\}$ and the associated antinorm
$$
\alpha(x, y) = \left\{
\begin{array}{rcl}
y - x, & \text{for} & x \geqslant 0, \, y \geqslant |x|,\\
\sqrt{y^2 - x^2}, & \text{for} & x < 0, \, y \geqslant |x|,\\
-\infty, & \text{for} & y < |x|.\\
\end{array}
\right.
$$
This antinorm is linear on the subcone $p^{\vee} = \{ (x,y) \, | \, y \geqslant x, \, x \geqslant 0\}$ for $p = (1,-1)$.
\end{example}

\begin{lemma}
\label{lemma-8}
If an antinorm $\alpha$ is linear on a subcone adjacent to the boundary of the cone $\Cone$, then
there exists $p \in \rb{(\Cone^{\vee})}$ such that $p|_{\Cone} \neq 0$, $p^{\vee}_1 \neq \varnothing$ and $\alpha^{\vee}(p) > 0$.
\end{lemma}

\begin{proof}
If $p \in \rb{(\Cone^{\vee})}$ and $p|_{\Cone} \neq 0$, then by Lemma~\ref{lemma-7} there exists $u \in \rb{(\Cone ^{\vee})} \setminus 0$ such that
$\langle p, u \rangle = 0$, that means $0 \neq u \in p^{\vee}_0 \subset p^{\vee}$. Moreover, the set $p^{\vee}$ is a closed convex cone by Lemma~\ref{lemma-4}.
Since the closure of the relative interior of a closed convex set coincides with the set itself~\cite[\S~6]{17},
then there exists a nonzero $v \in \ri{(p^{\vee})}$ (indeed, otherwise $p^{\vee} = 0$).
If at the same time $\ri{(p^{\vee})} \subset \ri{\Cone}$, then $v \in \ri{\Cone}$, whence automatically $\alpha(v) > 0$.
By Lemma~\ref{lemma-4} we have $\bar{v} = \frac{v}{\alpha(v)} \in p^{\vee}$ and $\alpha(\bar{v}) = 1$, that is $v \in p^{\vee}_1 \neq \varnothing$.
Then, by the definition of the dual function $\alpha^{\vee}(p) = -\langle p, \bar{v} \rangle$, that is greater than zero by Lemma~\ref{lemma-7}.
\end{proof}

\begin{remark}
\label{rem-linear}
If the antinorm $\alpha$ is linear on a subcone adjacent to the boundary, then the condition of Theorem~\ref{th-1} is violated,
namely $\alpha^{\vee}|_{\rb{(\Cone^{\vee})}} \neq $0.
Indeed, by Lemma~\ref{lemma-8} there exists $p \in \rb{(\Cone^{\vee})}$ such that $\alpha^{\vee}(p) > 0$.
\end{remark}

\begin{proposition}
\label{prop-1}
If the continuous antinorm $\alpha$ is linear on a subcone adjacent to the boundary of the cone $\Cone$,
then there is a normal extremal trajectory of problem~\eqref{podobryaev-eq-optimal-control-problem},
that has both time-like and light-like tangent vectors.
\end{proposition}

\begin{proof}
Due to Lemma~\ref{lemma-8} there exists $p \in \rb{(\Cone^{\vee})}$ such that $p^{\vee}_1 \neq \varnothing$.
Then by Lemma~\ref{lemma-6} we have
$$
\max\limits_{u \in p^{\vee} \setminus 0}{(\langle p, u \rangle + \alpha(u))} = \max\limits_{\mu \geqslant 0}{\mu(\langle p, \bar{u} \rangle + 1)}, \qquad
\bar{u} \in p^{\vee}_1.
$$
If $\langle p, \bar{u} \rangle + 1 < 0$, then the maximum is achieved at $\mu = 0$ and the control $u_p$ is light-like.
If $\langle p, \bar{u} \rangle + 1 = 0$, then $\mu$ can be any nonnegative number and the control $u_p$ can be light-like or time-like.
If $\langle p, \bar{u} \rangle + 1 > 0$, then there is no maximum.

Consider the trajectory of the conjugate subsystem passing through the point $p$. If $u(0)$ is timelike/lightlike, then when $u_p$ is chosen to be lightlike/timelike, the corresponding extremal trajectory has both time-like and light-like tangent vectors.
\end{proof}

Let us consider the following example where the dual function $\alpha^{\vee}$ is also not an antinorm,
but in contrast to the situation of Proposition~\ref{prop-1} the causal type of control is uniquely determined.

\begin{example}
\label{ex-2}
Consider the Heisenberg group, i.e., the space $\R^3$ with coordinates $a, b, c$ and the multiplication law
$$
(a_1,b_1,c_1) \cdot (a_2, b_2, c_2) = (a_1+a_2, b_1+b_2, c_1+c_2 + a_1b_2 - a_2b_1).
$$
The corresponding Lie algebra is the space $\R^3$ with coordinates $x,y,z$.
Consider the left-invariant sub-Lorentzian structure defined by the following cone and antinorm:
$$
\Cone = \{(x,y,z) \, | \, x,y \geqslant 0, \, z = 0\}, \qquad \alpha(x, y) = xy/(x+y), \ x^2 + y^2 \neq 0.
$$
Then the dual function $\alpha^{\vee}$ on the dual cone takes the form:
$$
\Cone^{\vee} = \{(h_1, h_2, h_3) \in (\R^3)^* \, | \, h_1, h_2 \leqslant 0\}, \qquad \alpha^{\vee}(h_1, h_2, h_3) = \left(\sqrt{|h_1|} + \sqrt{|h_2|}\right)^2.
$$
Note that $\alpha^{\vee}|_{\rb{(\Cone^{\vee})}} \neq 0$. It follows from Lemma~\ref{lemma-2} that the condition of Theorem~\ref{th-1} is not satisfied.
One can easily check that the conjugate subsystem of the Hamiltonian system read as:
$$
\left\{
\begin{array}{l}
\dot{h}_1 = -h_3u_2,\\
\dot{h}_2 = h_3u_1,\\
\dot{h}_3 = 0,
\end{array}
\right.
\qquad
u_h = (u_1, u_2) = \left\{
\begin{array}{lll}
\left(\sqrt{h_2/h_1}+1, \sqrt{h_1/h_2} + 1\right), & \text{for} & h_1h_2 \neq 0,\\
(a, 0), & \text{for} & h_1 = 0,\\
(0, b), & \text{for} & h_2 = 0,\\
\end{array}
\right.
$$
where $u_h$ is the control corresponding to the covector $h = (h_1, h_2, h_3) \neq 0$, and the numbers $a, b > 0$.
It is easy to prove, that for any normal extremal trajectory there are not more than two switches between time-like and light-like controls.
For example, the extremal trajectory with the initial covector $(0,-2,1)$ is the union of light-like, time-like and light-like arcs.
\end{example}

The following corollary of Theorem~\ref{th-1} allows us to pass to the "energy" to derive equations for extremal trajectories in (sub-)Lorentzian problems, where the antinorm is determined by a quadratic form. We will say that trajectories \emph{geometrically coincide} if their images as functions of time coincide one with another.

\begin{corollary}
\label{crl-1}
Assume that the antinorm $\alpha$ is defined by a quadratic form $q$ of signature $(1, r)$ on the Lie algebra $\g$, i.e., $\alpha(u) = \sqrt{q(u)}$,
where for some basis $e_0, e_1, \dots, e_r, \dots, e_n$ of the Lie algebra $\g$ we have $q(u) = u_0^2 - u_1^2 - \dots - u_r^2$.
Then normal extremal trajectories of problem~\emph{\eqref{podobryaev-eq-optimal-control-problem}}
geometrically coincide with normal time-like and light-like extremal trajectories of the same control system with quadratic functional
$$
\frac{1}{2} \int\limits_0^{t_1}{\left(u_0^2(t) - u_1^2(t) - \dots - u_r^2(t)\right) \, dt} \rightarrow \max,
$$
where the terminal time $t_1$ is fixed and $u \in L^{\infty}([0,t_1], \g \setminus 0)$ is a control.
\end{corollary}

\begin{proof}
Introduce the notation $h_i = \langle \,\cdot\, , e_i \rangle$, $i=0,1,\dots,r$ for linear on the space $\g^*$ Hamiltonians.
Applying Theorem~\ref{th-1} to problem~\eqref{podobryaev-eq-optimal-control-problem} we get
$$
\Cone^{\vee} = \{h = (h_0,h_1,\dots,h_n) \in \g^* \, | \, h_0^2 \geqslant h_1^2 + \dots + h_r^2, \ h_0 \leqslant 0 \},
$$
$$
\alpha^{\vee}(h) = \sqrt{h_0^2 - h_1^2 - \dots - h_r^2}, \qquad
u_h = \mu (-h_0 e_0 + h_1 e_1 + \dots + h_r e_r), \quad \mu > 0,
$$
$$
\dot{h}_i(t) = u_0(t)\{h_0(t),h_i(t)\} + u_1(t)\{h_1(t), h_i(t)\} + \dots + u_r(t)\{h_r(t), h_i(t)\}.
$$
In the time-like case we may assume up to parametrisation that $\mu = 1$, since the function $\alpha(u_{h(t)})$ is separated from zero on the segment $[0,t_1]$.

On the other hand, extremals $\lambda(\cdot)$ of the considered control system with the quadratic functional are defined by two first conditions~\eqref{podobryaev-eq-pmp}, where
$$
H_u^{\nu}(\lambda(t)) = u_0(t)h_0(t) + u_1(t)h_1(t) + \dots + u_r(t)h_r(t) + \frac{\nu}{2}\left(u_0^2(t) - u_1^2(t) - \dots - u_r^2(t)\right).
$$
From the maximum condition for $\nu = 1$ it follows that
$u_0(t) = -h_0(t)$, $u_1(t) = h_1(t)$, \dots, $u_r(t) = h_r(t)$,
and the maximized Hamiltonian is equal to
$$
H = -\tfrac{1}{2}\left(h_0^2 - h_1^2 - \dots - h_r^2\right).
$$
The corresponding conjugate subsystem reads as:
$$
\dot{h}_i(t)= \{H, h_i(t)\} = -h_0(t)\{h_0(t). h_i(t)\} + h_1(t)\{h_1(t), h_i(t)\} + \dots + h_r(t)\{h_r(t), h_i(t)\}.
$$
Hence, the right sides of the conjugate subsystems of the both problems coincide one with another.
For a smooth Hamiltonian there is a unique solution of the Cauchy problem for the Hamiltonian system.
So, an extremal trajectory is defined by its initial covector, i.e., a point of the space $\g^* = T_{\id}^*G$.
It remains to note that the Hamiltonian $H$ is a first integral of the conjugate subsystem.
Covectors of normal time-like extremal trajectories with natural parametrisation (i.e., such that $q(u(t)) = 1$) are located at the level surface of the Hamiltonian $H = -1/2$, $h_0 < 0$, that coincides with the set $S_1^{\vee}$.
Covectors of normal light-like extremal trajectories are located on the level surface $H = 0$, $h_0 < 0$, that coincides with the set $S_0^{\vee}$.
\end{proof}

\begin{remark}
\label{rem-quadratic}
Under the conditions of Corollary~\ref{crl-1} it is sometimes convenient to consider a basis in which the quadratic form has the form
$$
q(u) = c_0 u_0^2 - c_1 u_1^2 - \dots - c_r u_r^2, \qquad \text{where} \qquad c_0, c_1, \dots, c_r > 0.
$$
The choice of such a basis is explained by its consistency with another quadratic form, for example, the Killing form. In this case
$$
H = -{{1}\over{2}}\Biggl(\frac{h_0^2}{c_0} - \frac{h_1^2}{c_1} - \dots - \frac{h_r^2}{c_r}\Biggr),
$$
and initial covectors of normal time-like extremal trajectories (with natural parametrisation) and light-like extremal trajectories are located on the level surfaces $H = -1/2$ and $H = 0$ (with the additional condition $h_0 < 0$), respectively.
\end{remark}

Generally speaking, the behavior of abnormal trajectories can be complex.
The corresponding trajectories of the conjugate subsystem lie on the relative boundary of the dual cone $\rb{(\Cone^{\vee})}$ and the nature of their intersection with the annihilator of the space $\sspan{\Cone}$ is not a priori clear.
However, in some cases abnormal trajectories can be described.

\begin{corollary}
\label{crl-2}
In the Lorentzian case, i.e., $\sspan{\Cone} = \g$, any abnormal extremal trajectory is light-like.
In particular, any abnormal extremal trajectory is not strictly abnormal.
\end{corollary}

Thus, abnormal trajectories always arise in Lorentzian geometry, in contrast to Riemannian geometry, where there are no abnormal trajectories.

\begin{definition}
\label{def-8}
A distribution $\Delta$ of subspaces on a three-dimensional smooth manifold $M$ is called \emph{contact}, if
there exists 1-form $\omega$ such that $\Delta_m = \Ker{\omega_m}$ for any point $m \in M$ and $\omega \wedge d\omega \neq 0$.
\end{definition}

\begin{corollary}
\label{crl-3}
If the distribution $L_{g *} \sspan{\Cone}$ is contact, then all abnormal extremal trajectories of the sub-Lorentzian problem~\eqref{podobryaev-eq-optimal-control-problem}
are light-like and in particular not strictly abnormal.
\end{corollary}

\begin{proof}[Proof of Corollaries~\emph{\ref{crl-2}--\ref{crl-3}}]
It is clear that sub-Riemannian abnormal extremals lie in the annihilator of the distribution $\Delta$. In the Lorentzian case the distribution $\Delta$ coincides with the whole tangent bundle. This proves Corollary~\ref{crl-2}.

It is known~\cite[\S~4.3]{8} that for distributions of the form $\Delta = \Ker{\omega}$ sub-Riemannian abnoraml extremals lie in the Martinet set $\{m \in M \, | \, (\omega \wedge d\omega)_m = 0\}$, that is empty in our case. This proves Corollary~\ref{crl-3}.
\end{proof}

Let us now consider an example of a non-contact sub-Lorentzian structure with step greater than two.

\begin{example}
\label{ex-3}
Consider a sub-Lorentzian structure on the free Carnot group $G$ of rank 2 and step 4. This is a connected and simply connected Lie group whose Lie algebra is linearly generated by the elements:
$$
X_1, \quad X_2, \quad X_3 = [X_1, X_2], \quad X_4 = [X_1, X_3], \quad X_5 = [X_2, X_3],
$$
$$
X_6 = [X_1, X_4], \quad X_7 = [X_1, X_5] = [X_2, X_4], \quad X_8 = [X_2, X_5],
$$
the rest of commutators of these vectors are equal to zero. Let $\Delta_g = L_{g *} \sspan{\{X_1, X_2\}} \subset T_gG$ be a two-dimensional distribution,
$U = \{(u_1, u_2) \in \R^2 \, | \, u_1^2 - u_2^2 \geqslant 0, \, u_1 > 0\}$ be a control set,
and dynamics is defined by the non-autonomous differential equation $\dot{g}(t) = L_{g(t) *} (u_1X_1 + u_2X_2)$.

Abnormal curves of the distribution $\Delta$ on the group $G$ are described in~\cite{18}.
Moreover, this is the simplest free Carnot group where strictly abnormal trajectories (for the distribution generated by the first layer of the corresponding Lie algebra) appear.
The conjugate subsystem in the coordinates $h_i = \langle \,\cdot\, , X_i \rangle$, $i=1,\dots, 8$ reads as:
$$
\begin{array}{lllll}
\dot{h}_1 = -u_2h_3, & \qquad & \dot{h}_3 = u_1h_4 + u_2h_5, & \qquad & \dot{h}_4 = u_1h_6 + u_2h_7,\\
\dot{h}_2 = u_1h_3, & \qquad & \dot{h}_6 = \dot{h}_7 = \dot{h}_8 = 0, & \qquad & \dot{h}_5 = u_1h_7 + u_2h_8.\\
\end{array}
$$
Let us consider the case $h_4 = h_5 = h_6 = h_7 = h_8 = 0$, then
$\dot{h}_1 = -u_2h_3$, $\dot{h}_2 = u_1h_3$, $\dot{h}_3 = 0$.
Thus, corresponding light-like extremal trajectories (for $h_3 \leqslant 0$) have the form:
\begin{equation}
\label{podobryaev-eq-corner}
g(t) = \left\{
\begin{array}{lll}
\exp{\bigl(t\alpha(t)(X_1+X_2)\bigr)} & \text{for} & t \leqslant \bar{t},\\
\exp{\bigl(\bar{t}(X_1+X_2)\bigr)}\cdot\exp{\bigl((t-\bar{t})k(t)(X_1-X_2)\bigr)} & \text{for} & t > \bar{t},\\
\end{array}
\right.
\end{equation}
where $\bar{t} \in [0, T]$ and a function $k \in L^{\infty}([0, T], \R_+)$ are arbitrary.
Generally speaking, the projections of these extremal trajectories to the plane $\sspan{\{X_1, X_2\}}$ look like corners.
These normal light-like extremal trajectories are at the same time abnormal. So, these trajectories are not strictly abnormal.

Moreover, there are abnormal sub-Lorentzian extremal trajectories that are abnormal curves for the distribution $\Delta$, in other words, sub-Riemannian abnormal trajectories.
Of course, not any arc of an abnormal curve of the distribution $\Delta$ is an arc of abnormal sub-Lorentzian extremal trajectory,
since the velocities of such arc must be admissible in the sub-Lorentzian sense.
It is proved in~\cite{18} that sub-Riemannian abnormal trajectories project into straight lines, corners or second-order curves on the plane $\sspan{\{X_1, X_2\}}$.
In particular, sub-Riemannian abnormal trajectories projecting into ellipses or parabolas cannot be admissible in the sub-Lorentzian sense.

On the other hand, even if a strictly abnormal sub-Riemannian trajectory is a sub-Lorentzian abnormal trajectory, it may not be strictly abnormal in the sub-Lorentzian sense.
For example, angles are strictly abnormal for the distribution (see~\cite{18}), but at the same time they are light-like normal sub-Lorentzian extremal trajectories.
\end{example}

\end{document}